\documentclass[10pt,a4paper]{amsart}
\pdfoutput=1
\usepackage[sec,skipmbb]{rune}
\usepackage{runediag}

\DeclareSymbolFont{bbold}{U}{bbold}{m}{n}
\DeclareSymbolFontAlphabet{\mathbbold}{bbold}
\newcommand{\bbone}{\mathbbold{1}}

\usepackage{eucal}

\title{Algebras for Enriched $\infty$-Operads}
\author{Rune Haugseng}
\address{Norwegian University of Science and Technology (NTNU),
  Trondheim, Norway}

\newcommand{\Dint}{\simp_{\txt{int}}}
\newcommand{\Dintop}{\Dint^{\op}}
\newcommand{\DF}{\bbDelta_{\mathbb{F}}}

\newcommand{\DFop}{\DF^{\op}}

\newcommand{\DFX}{\bbDelta_{\mathbb{F},X}}
\newcommand{\DFXop}{\DFX^{\op}}
\newcommand{\xF}{\mathbb{F}}
\newcommand{\xFeq}{\xF^{\simeq}}

\newcommand{\Fact}{\txt{Fact}}

\DeclareMathOperator{\Tw}{Tw}

\newcommand{\opd}{\txt{opd}}

\newcommand{\COLL}{\txt{COLL}}
\newcommand{\Coll}{\txt{Coll}}
\newcommand{\RM}{\txt{RM}}

\newcommand{\drpullback}{\arrow[phantom]{dr}[very near start,description]{\lrcorner}}

\newcommand{\Algd}{\txt{Algd}}

\newcommand{\MAP}{\txt{MAP}}
\newcommand{\HOM}{\txt{HOM}}

\usepackage{graphicx}

\DeclareMathOperator{\ob}{ob}

\begin{document}

\begin{abstract}
  Using the description of enriched $\infty$-operads as associative
  algebras in symmetric sequences, we define algebras for enriched
  $\infty$-operads as certain modules in symmetric sequences. For
  $\mathbf{V}$ a symmetric monoidal model category and $\mathbf{O}$ a
  $\Sigma$-cofibrant operad in $\mathbf{V}$ for which the model
  structure on $\mathbf{V}$ can be lifted to one on
  $\mathbf{O}$-algebras, we then prove that strict algebras in
  $\mathbf{V}$ are equivalent to $\infty$-categorical algebras in the
  symmetric monoidal $\infty$-category associated to $\mathbf{V}$. We
  also show that for an $\infty$-operad $\mathcal{O}$ enriched in a
  suitable closed symmetric monoidal $\infty$-category $\mathcal{V}$,
  we can equivalently describe $\mathcal{O}$-algebras in $\mathcal{V}$
  as morphisms of \iopds{} from $\mathcal{O}$ to a self-enrichment of
  $\mathcal{V}$.
\end{abstract}

\maketitle

\tableofcontents

\section{Introduction}
If $\mathbf{V}$ is a symmetric monoidal category whose tensor product
is compatible with colimits, then (one-object\footnote{We focus on the
  one-object case for simplicity, but similar descriptions apply to
  ($\infty$-)operads with any fixed set (space) of objects.})
\emph{operads} enriched in $\mathbf{V}$ can be described as
associative algebras in $\Fun(\xF^{\simeq}, \mathbf{V})$, the category
of \emph{symmetric sequences} in $\mathbf{V}$. Here $\xF^{\simeq}$
denotes the groupoid $\coprod_{n} B\Sigma_{n}$ of finite sets and
bijections, and the monoidal structure on symmetric sequences is the
\emph{composition product}, which is a monoidal structure given by the
formula
\[(X \odot Y)(n) \cong \coprod_{k =0}^{\infty} \left(
    \coprod_{i_{1}+\cdots+i_{k} = n} (Y(i_{1}) \otimes \cdots \otimes
    Y(i_{k})) \times_{\Sigma_{i_{1}} \times \cdots \times
      \Sigma_{i_{k}}} \Sigma_{n} \right) \otimes_{\Sigma_{k}} X(k).\]
In a previous paper \cite{symmseq} we proved that (one-object) \iopds{} enriched in
a suitable symmetric monoidal \icat{} $\mathcal{V}$ admit a similar
description, as associative algebras in $\Fun(\xF^{\simeq},
\mathcal{V})$ using a monoidal structure given by the same formula.

Our goal in this short paper is to use this description of \iopds{} to
study \emph{algebras} for enriched \iopds{}. Classically, if $\mathbf{O}$ is a
(one-object) $\mathbf{V}$-operad, then an \emph{$\mathbf{O}$-algebra}
in $\mathbf{V}$ consists of an object $A \in \mathbf{V}$ and
$\Sigma_{n}$-equivariant morphisms
\[  A^{\otimes n} \otimes \mathbf{O}(n)  \to A\] compatible with the
composition and unit of $\mathbf{O}$. This data can be packaged in a
convenient way using the composition product: an $\mathbf{O}$-algebra
is the same thing as a right\footnote{This is correct under our
  convention for the ordering of the composition product, chosen to be
  compatible with our construction of the \icatl{} version; in most
  references on ordinary operads the reverse ordering is used, so that
  $\mathbf{O}$-algebras are certain \emph{left} $\mathbf{O}$-modules.}
$\mathbf{O}$-module $M$ in $\Fun(\xF^{\simeq}, \mathbf{V})$ that is
concentrated in degree zero, \ie{} $M(n)$ is the initial object $\emptyset$ for
$n \neq 0$. Indeed, such a right $\mathbf{O}$-module is given by a
morphism
\[ M \odot \mathbf{O} \to M,\] and expanding out the composition
product we see that (since $M(n)$ vanishes for $n \neq 0$) this is
precisely given by a map
\[ \coprod_{k} M(0)^{\otimes k} \otimes_{\Sigma_{k}} \mathbf{O}(k) \to M(0).\]

Here we take the corresponding modules in the \icatl{} setting (and
their analogues for many-object operads) as a \emph{definition} of
algebras for enriched \iopds{}. For a symmetric monoidal \icat{}
$\mathcal{V}$ (whose tensor product is compatible with colimits
indexed by \igpds{}) and a $\mathcal{V}$-enriched \iopd{}
$\mathcal{O}$, this results in an \icat{}
$\Alg_{\mathcal{O}}(\mathcal{V})$ with several pleasant properties,
including the expected formula for free $\mathcal{O}$-algebras, as we
will see in \S\ref{subsec:algmod} after reviewing the results of
\cite{symmseq} in \S\ref{sec:review}.

We then prove two main results about this \icatl{} notion of $\mathcal{O}$-algebras.
First, in \S\ref{subsec:comp} we prove a rectification result for algebras
over operads enriched in a symmetric monoidal model category:
\begin{thm}[See \cref{thm:opdalgrect}]
  Let $\mathbf{V}$ be a symmetric monoidal model category (with
  cofibrant unit) and $\mathbf{O}$ an $S$-coloured $\Sigma$-cofibrant
  $\mathbf{V}$-operad such that the category
  $\Alg_{\mathbf{O}}(\mathbf{V})$ admits a model structure whose weak
  equivalences and fibrations are detected by the forgetful functor to
  $\Fun(S, \mathbf{V})$. Then this model category describes the
  \icat{} of algebras for this operad in the \icatl{} localization of
  $\mathbf{V}$. That is, there is an
  equivalence of \icats{}
  \[ \Alg_{\mathbf{O}}(\mathbf{V})[W_{\mathbf{O}}^{-1}] \simeq
    \Alg_{\mathcal{O}}(\mathcal{V}) \]
  where on the left $W_{\mathbf{O}}$ denotes the collection of weak equivalences
  between $\mathbf{O}$-algebras, and on the right $\mathcal{V} :=
  \mathbf{V}[W^{-1}]$ is the localization of $\mathbf{V}$ at its weak
  equivalences and $\mathcal{O}$ denotes
  $\mathbf{O}$ viewed an \iopd{} enriched in $\mathcal{V}$ via the
  localization functor.
\end{thm}
The comparison applies, for instance, to \emph{all} $\Sigma$-cofibrant
operads in chain complexes over a field of characteristic zero or in
simplicial sets.\footnote{For more general model categories, including
  chain complexes in positive characteristic, there is typically only
  a \emph{semi}-model structure on algebras over a $\Sigma$-cofibrant
  operad; see \cref{rmk:semimodel} for more discussion of this case.}
We in fact prove a slightly more general result that avoids the
assumption that the unit is cofibrant, which applies to all
$\Sigma$-cofibrant operads in symmetric spectra. The proof boils down
to a combination of model-categorical results of Pavlov--Scholbach and
our formula for free algebras, using the same strategy as
\cite{HA}*{Theorems 4.1.4.4} and \cite{PavlovScholbachOpd}*{Theorem
  7.10} to prove that both sides are \icats{} of algebras for
equivalent monads.

Another classical description of algebras over (one-object)
$\mathbf{V}$-operads uses \emph{endomorphism operads}: For $v$ an
object of $\mathbf{V}$ there is an operad $\End_{\mathbf{V}}(v)$ with
$n$-ary operations given by the internal Hom
$\HOM_{\mathbf{V}}(v^{\otimes n}, v)$ (where the $\Sigma_{n}$-action
permutes the factors in $v^{\otimes n}$). If $\mathbf{O}$ is a
(one-object) $\mathbf{V}$-operad then we can describe
$\mathbf{O}$-algebras in $\mathbf{V}$ with underlying object $v$ as
morphisms of one-object operads
$\mathbf{O} \to \txt{End}_{\mathbf{V}}(v)$. More generally, we can
consider an $S$-coloured endomorphism operad $\End_{\mathbf{V}}(f)$
for any map of sets $S \to \ob \mathbf{V}$, where operations from
$(s_{1},\ldots,s_{n})$ to $s'$ are given by
$\HOM_{\mathbf{V}}(f(s_{1}) \otimes \cdots \otimes f(s_{n}), f(s'))$;
for an $S$-coloured operad $\mathbf{O}$, we can then identify
$\mathbf{O}$-algebras in $\mathbf{V}$ given by $f$ on objects with
morphisms of $S$-coloured operads from $\mathbf{O}$ to
$\End_{\mathbf{V}}(f)$.

In \S\ref{subsec:endo} we will use Lurie's construction of
endomorphism algebras \cite{HA}*{\S 4.7.1}, following work of Hinich
\cite{HinichYoneda} in the case of enriched \icats{}, to construct
endomorphism \iopds{} $\End_{\mathcal{V}}(f)$ for any map of \igpds{}
$f \colon X \to \mathcal{V}^{\simeq}$, where
 $\mathcal{V}$ is a
closed symmetric monoidal \icat{} compatible with colimits indexed by
small \igpds{}. Moreover, we show that these endomorphism \iopds{} can
be combined into a self-enrichment of $\mathcal{V}$, which gives our
second main result:
\begin{thm}[See \cref{thm:oVmapeq}]
  Let $\mathcal{V}$ be a closed symmetric monoidal \icat{} compatible
  with colimits indexed by small \igpds{}. Then there exists a 
  $\mathcal{V}$-\iopd{} $\overline{\mathcal{V}}$, whose object of
  multimorphisms from $(v_{1},\ldots,v_{n})$ to $w$ is the internal
  Hom $\MAP_{\mathcal{V}}(v_{1} \otimes \cdots \otimes v_{n}, w)$, 
  such that for a $\mathcal{V}$-\iopd{} $\mathcal{O}$
  there is a natural equivalence of \igpds{}
  \[ \{\text{$\mathcal{O}$-algebras in $\mathcal{V}$}\} \simeq \{\text{morphisms of $\mathcal{V}$-\iopds{}
     $\mathcal{O} \to \overline{\mathcal{V}}$}\}. \]
\end{thm}

\begin{warning}
  Throughout this paper we use the term ``$\mathcal{V}$-\iopd{}'' to
  refer to the \emph{algebraic} notion of an \iopd{}, given by objects
  of multimorphisms in $\mathcal{V}$ with homotopy-coherently
  associative and unital composition operations. Thus we have a class
  of fully faithful and essentially surjective morphisms between
  enriched \iopds{} that we would have to invert to get the
  ``correct'' \icat{} of $\mathcal{V}$-\iopds{}. In terms of the
  description of $\mathcal{V}$-\iopds{} we use, this means we are not
  requiring these to be ``complete'' (see \cite{enropd}*{\S 3}). In
  the terminology of \cite{AyalaFrancisFlagged}, our enriched \iopds{}
  can be thought of as being \emph{flagged} enriched \iopds{}, meaning
  a (complete) enriched \iopd{} equipped with an essentially
  surjective morphism of \igpds{} to its space of objects. Note that,
  as we will see in \cref{rmk:ffeseqonalgs}, the \icats{} of algebras
  we will study are in fact invariant under fully faithful and
  essentially surjective maps of enriched \iopds{}, so it does not
  really make a difference whether we use complete objects or not.
\end{warning}

\subsection{Related Work}
Much of our work here is not particularly reliant on the specific
construction of the composition product from \cite{symmseq}. An
alternative construction, using the description of symmetric sequences
in $\mathcal{V}$ as the free presentably symmetric monoidal \icat{} on
$\mathcal{V}$ and generalizing the approach to 1-categorical operads
due to Trimble \cite{TrimbleLie} and Carboni, has been worked out by
Brantner~\cite{BrantnerThesis}; however, this construction of \iopds{}
has not yet been compared to any of the other approaches.  In the
setting of dendroidal sets, Heuts describes algebras valued in spaces
and \icats{} in terms of dendroidal versions of left and cocartesian
fibrations in \cite{HeutsMTh}.

\subsection{Acknowledgments}
I thank Stefan Schwede and Irakli Patchkoria for helpful discussions
about model structures on spectra, and David White for help with
model structures on operad algebras. Much of this paper was
written while the author was employed by the IBS Center for Geometry
and Physics in a position funded by grant IBS-R003-D1 of the Institute
for Basic Science of the Republic of Korea.

\section{$\infty$-Operads as Algebras}\label{sec:review}
In this section we will review the main results on enriched \iopds{}
from \cite{symmseq}, 
where we showed that enriched \iopds{} can be viewed as associative
algebras in a double \icat{} of symmetric collections. For this the
relevant notion of enriched \iopds{} is an enriched variant of Barwick's
definition of \iopds{} \cite{BarwickOpCat} as presheaves on a
category $\DF$ satisfying Segal (and completeness) conditions; this
definition was first introduced in \cite{enropd}. We will now briefly
review this definition, as well as a slight generalization considered in
\cite{symmseq} that we will use to define modules over enriched
\iopds{}; we start by recalling the definition of Barwick's category $\DF$:

\begin{defn}
  Let $\mathbb{F}$ denote a skeleton of the category of finite sets,
  with objects $\mathbf{k} := \{1,\ldots,k\}$, $k = 0,1,\ldots$. We write
  $\DF$ for the category whose objects are pairs
  $([n], f \colon [n] \to \mathbb{F})$, with a morphism
  $([n], f) \to ([m], g)$ given by a morphism
  $\phi \colon [n] \to [m]$ in $\simp$ and a natural transformation
  $\eta \colon f \to g \circ \phi$ such that
  \begin{enumerate}[(i)]
  \item the map $\eta_{i} \colon f(i) \to g(\phi(i))$ is injective for
    all $i = 0,\ldots,m$,
  \item the commutative square
    \[\csquare{f(i)}{g(\phi(i))}{f(j)}{g(\phi(j))}{\eta_{i}}{}{}{\eta_{j}}\]
    is cartesian for all $0 \leq i \leq j \leq m$.
  \end{enumerate}
  Note that there is an obvious projection $\DFop \to \Dop$; this is a
  double \icat{}. There is also a functor $V \colon \DFop \to \xF_{*}$
  which takes $([n], f)$ to $(\coprod_{i=1}^{n} f(i))_{+}$; see
  \cite{enropd}*{Definition 2.2.11} for a complete definition.
  \end{defn}

\begin{remark}
  An object of $\DF$ is a sequence of maps of finite sets
  \[ \mathbf{a}_{0} \xto{f_{1}} \mathbf{a}_{1} \to \cdots \xto{f_{n}}
    \mathbf{a}_{n}, \]
  which we can think of as a forest of $|\mathbf{a}_{n}|$ oriented trees with
  $n$ levels: the edges are the elements of all the sets
  $\mathbf{a}_{i}$, and since each vertex has a unique outgoing edge
  we can also think of the elements of $\mathbf{a}_{i}$ for $i > 0$ as
  the vertices; the function $f_{i}$ assigns to each edge in level
  $i-1$ the vertex of which it is an incoming edge. Note that this
  means the functor
  $V$ takes each forest to its set of vertices (with a disjoint base
  point). The morphisms in $\DF$ are defined so that a vertex in the
  source is mapped to a subtree of the target with the same number of
  incoming edges.
\end{remark}

\begin{defn}
  For $X \in \mathcal{S}$, we write $\DFXop \to \DFop$ for the left
  fibration corresponding to the functor $\DFop \to \mathcal{S}$
  obtained as the right Kan extension of the functor $* \to
  \mathcal{S}$ with value $X$ along the inclusion $\{([0],
  \mathbf{1})\} \hookrightarrow \DFop$.
\end{defn}

\begin{remark}
  The fibre of $\DFXop \to \DFop$ at an object $F$ is equivalent to a
  product of copies of $X$ indexed by the number of edges in the forest
  $F$; we can thus think of an object of $\DFXop$ as a forest whose
  edges are labelled by points of $X$.
\end{remark}

\begin{notation}
  If $\mathcal{O}$ is a \nsiopd{}\footnote{We do not review the
    definition here, as it will not really play a role in this paper:
    the only examples we will encounter are $\Dop$ and the
    non-symmetric operad for right modules, which we describe
    explicitly in \cref{defn:RMopd}. We refer the reader to
    \cite{symmseq}*{\S 2.1} for a brief review, or \cite{enr}*{\S 2.2}
    for more motivation and discussion of the definition.}, we write
  $\mathcal{O}_{\xF} := \mathcal{O} \times_{\Dop} \DFop$ and
  $\mathcal{O}_{\xF,X} := \mathcal{O} \times_{\Dop} \DFXop$.
\end{notation}

\begin{defn}
  We say a morphism in $\DFop$ is \emph{operadic inert} if it lies
  over an inert morphism in $\Dop$. We then call a morphism in
  $\DFXop$ \emph{operadic inert} if it is a (necessarily cocartesian) morphism over
  an operadic inert morphism in $\DFop$. If $\mathcal{O}$ is a
  \nsiopd{}, we similarly say a morphism in $\mathcal{O}_{\xF,X}$ is
  \emph{operadic inert} if it maps to an
  inert morphism in $\mathcal{O}$ and an operadic inert morphism in
  $\DFop$. The functor $V \colon \DFop \to \xF_{*}$ takes operadic
  inert morphisms to inert morphisms, hence if $\mathcal{V}$ is a
  symmetric monoidal \icat{} we can define an
  \emph{operadic algebra} for $\mathcal{O}_{\xF,X}$ in $\mathcal{V}$ to be a commutative square
  \[
    \begin{tikzcd}
      \mathcal{O}_{\xF,X} \arrow{r}{A} \arrow{d} & \mathcal{V}^{\otimes}
  \arrow{d} \\
  \DFop \arrow{r}{V} & \xF_{*},
    \end{tikzcd}
  \]
  such that $A$ takes operadic inert morphisms to inert morphisms in
  $\mathcal{V}^{\otimes}$. We write
  $\Alg^{\opd}_{\mathcal{O}_{\xF,X}}(\mathcal{V})$ for the full
  subcategory of
  $\Fun_{/\xF_{*}}(\mathcal{O}_{\xF,X}, \mathcal{V}^{\otimes})$
  spanned by the operadic algebras. We also write
  $\Algd^{\opd}_{\mathcal{O}_{\xF}}(\mathcal{V}) \to \mathcal{S}$ for
  the cartesian fibration corresponding to the functor
  $X \mapsto \Alg^{\opd}_{\mathcal{O}_{\xF,X}}(\mathcal{V})$ and refer
  to its objects as \emph{operadic $\mathcal{O}_{\xF}$-algebroids} in
  $\mathcal{V}$.
\end{defn}

\begin{remark}
  The condition for a commutative triangle
  \[
    \begin{tikzcd}
      \DFXop \ar[rr, "F"] \ar[dr] & & \mathcal{V}^{\otimes} \ar[dl] \\
      & \xF_{*}
    \end{tikzcd}
  \]
  to be an operadic algebra is essentially that the value of $F$ at a
  forest whose edges are labelled by points of $X$ consists of a list
  of the values of $F$ at the corollas (one-vertex subtrees) of the
  forest. The latter should be thought of as the objects of
  multimorphisms in a $\mathcal{V}$-enriched \iopd{}, whose
  homotopy-coherent composition is encoded by the rest of the data in
  $F$. Indeed, operadic algebras for $\DFXop$ gives one of the notions
  of enriched \iopds{} introduced in \cite{enropd}, which justifies
  the following notation:
\end{remark}

\begin{notation}
  For a space $X$, we write
  $\Opd_{X}(\mathcal{V}):= \Alg^{\opd}_{\DFXop}(\mathcal{V})$. We also
  write $\Opd(\mathcal{V}) := \Algd^{\opd}_{\DFop}(\mathcal{V})$, so
  that we have a cartesian fibration
  $\Opd(\mathcal{V}) \to \mathcal{S}$ whose fibre at $X$ is
  $\Opd_{X}(\mathcal{V})$.
\end{notation}

\begin{notation}
  For $X,Y \in \mathcal{S}$, we write  $\xFeq_{X,Y}$ for the \igpd{}
  $\coprod_{n = 0 }^{\infty} X^{\times n}_{h\Sigma_{n}} \times Y$. For
  a functor $\Phi \colon \xFeq_{X,Y} \to \mathcal{V}$ we will denote
  its value at $((x_{1},\ldots,x_{n}), y)$ by
  $\Phi\binom{x_{1},\ldots,x_{n}}{y}$. We
  also abbreviate $\xFeq_{X} := \xFeq_{X,X}$ and write
  $\Coll_{X}(\mathcal{V}):= \Fun(\xFeq_{X}, \mathcal{V})$; we refer to
  the objects of this \icat{} as (symmetric) \emph{$X$-collections} in
  $\mathcal{V}$.
\end{notation}

To state the main result we will use from \cite{symmseq}, we need to recall
one further definition:
\begin{defn}
  A \emph{double \icat{}} is a cocartesian fibration $\mathcal{F} \to
  \Dop$ that corresponds to a functor $F \colon \Dop \to \CatI$ that
  satisfies the \emph{Segal condition}: for every $n$, the functor
  \[ F([n]) \to F([1]) \times_{F([0])} \cdots \times_{F([0])} F([1]),\]
  induced by the value of $F$ at the inert maps $[0],[1] \to [n]$ in
  $\simp$,
  is an equivalence. We say a double \icat{} is \emph{framed} if the
  functor $F([1]) \to F([0]) \times F([0])$, induced by the two face
  maps $[0] \to [1]$, is a cocartesian fibration. If $\mathcal{O}$ is
  a non-symmetric \iopd{}, then an \emph{$\mathcal{O}$-algebra} in
  $\mathcal{F}$ is a commutative triangle
  \[
    \begin{tikzcd}
     \mathcal{O} \ar[rr, "A"] \ar[dr] & &  \mathcal{F} \ar[dl] \\
      & \Dop 
    \end{tikzcd}
  \]
  such that $A$ preserves cocartesian morphisms over inert maps in $\Dop$.
\end{defn}

\begin{remark}
  If $\mathcal{F} \to \Dop$ is a double \icat{}, then we think of this
  as an \icatl{} version of a double category where
  \begin{itemize}
  \item the objects are the objects of $\mathcal{F}_{[0]}$,
  \item the vertical morphisms are the morphisms in
    $\mathcal{F}_{[0]}$,
  \item the horizontal morphisms are the objects of
    $\mathcal{F}_{[1]}$,
  \item the squares are the morphisms in $\mathcal{F}_{[1]}$.
  \end{itemize}
\end{remark}

\begin{thm}[See \cite{symmseq}*{Corollary 4.2.8}]\label{sseqthm}
  Suppose $\mathcal{V}$ is a symmetric monoidal \icat{} compatible
  with colimits indexed by small \igpds{}. Then there exists a framed double
  \icat{} $\COLL(\mathcal{V})$ such that:
  \begin{enumerate}[(i)]
  \item $\COLL(\mathcal{V})_{0} \simeq \mathcal{S}$, \ie{} the objects
    of $\COLL(\mathcal{V})$ are small \igpds{} and the vertical morphisms are
    morphisms thereof.
  \item A horizontal morphism from $X$ to $Y$ is a functor
    $\xFeq_{X,Y} \to \mathcal{V}$.
  \item If $\Phi$ is a horizontal morphism from $X$ to $Y$ and $\Psi$
    is one from $Y$ to $Z$ then their composite $\Phi \odot_{Y} \Psi$
    is given by the formula
    \[ \Phi \odot_{Y} \Psi \binom{x_{1},\ldots,x_{n}}{z} \simeq
      \colim_{\mathbf{n} \to \mathbf{m} \to \mathbf{1}} \colim_{(y_{i}) \in Y^{\times m}} \bigotimes_{i \in \mathbf{m}}
        \Phi\binom{x_{k}: k \in \mathbf{n}_{i}}{y_{i}} \otimes
        \Psi\binom{y_{1},\ldots,y_{k}}{z}.\]
    \item If $\mathcal{O}$ is any \nsiopd{} then there is a natural
      equivalence
      \[ \Alg_{\mathcal{O}}(\COLL(\mathcal{V})) \simeq
        \Algd^{\opd}_{\mathcal{O}_{\xF}}(\mathcal{V}).\]
    \item If $F \colon \mathcal{V} \to \mathcal{W}$ is a symmetric
      monoidal functor that preserves colimits indexed by small \igpds{},
      then composition with $F$ induces a morphism of double \icats{}
      $\COLL(\mathcal{V}) \to \COLL(\mathcal{W})$.
  \end{enumerate}
\end{thm}
\begin{remark}
  In (iii), the outer colimit is more precisely over the groupoid
  $\Fact(\mathbf{n}\to \mathbf{1})$ of factorizations
  $\mathbf{n}\to \mathbf{m} \to \mathbf{1}$, with morphisms given by
  diagrams
  \[
    \begin{tikzcd}
      {} & \mathbf{m} \arrow{dd}{\sim} \arrow{dr} \\
      \mathbf{n} \arrow{ur} \arrow{dr} & & \mathbf{1}. \\
       & \mathbf{m}' \arrow{ur}
    \end{tikzcd}
  \]
\end{remark}

\begin{remark}\label{rmk:AlgCOLLiopd}
  In particular, associative algebras in $\COLL(\mathcal{V})$ are
  equivalent to \iopds{} enriched in $\mathcal{V}$:
  \[ \Alg_{\Dop}(\COLL(\mathcal{V})) \simeq 
    \Algd^{\opd}_{\DFop}(\mathcal{V}) \simeq \Opd(\mathcal{V}).\]
\end{remark}

For $X \in \mathcal{V}$, the \icat{}
\[\COLL(\mathcal{V})(X,X) \simeq
  \Coll_{X}(\mathcal{V}) \simeq \Fun(\xFeq_{X}, \mathcal{V})\] of
horizontal endomorphisms of $X$ has a monoidal structure given by
composition.  Moreover, by \cite{symmseq}*{Proposition 3.4.8} a
morphism $f \colon X \to Y$ induces a natural lax monoidal functor
$f^{*}\colon \Coll_{Y}(\mathcal{V}) \to \Coll_{X}(\mathcal{V})$, given
by composition with the induced map $\xFeq_{X} \to \xFeq_{Y}$.
By \cite{symmseq}*{Corollary 3.4.10} we also have:
\begin{cor}\label{wcopdfib}
  Let $\mathcal{O}$ be a weakly contractible \nsiopd{}. Then the
  functor \[\Alg_{\mathcal{O}}(\COLL(\mathcal{V})) \to \mathcal{S}\]
  given by evaluation at $* \in \mathcal{O}_{0}$ is a cartesian
  fibration corresponding to the functor $\mathcal{S} \to \CatI$ that
  takes $X$ to $\Alg_{\mathcal{O}}(\Coll_{X}(\mathcal{V}))$ and a
  morphism $f \colon X \to Y$ to the functor given by composition with
  the lax monoidal functor
  $f^{*}\colon \Coll_{Y}(\mathcal{V}) \to \Coll_{X}(\mathcal{V})$.
\end{cor}

\begin{remark}
  This corollary applies in particular to the weakly contractible
  non-symmetric \iopd{} $\Dop$, so that by \cref{rmk:AlgCOLLiopd},
  enriched \iopds{} with $X$ as space of objects are given by
  associative algebras in $\Coll_{X}(\mathcal{V})$, \ie{}
  \[ \Opd_{X}(\mathcal{V}) \simeq \Alg_{\Dop}(\Coll_{X}(\mathcal{V})).\]
\end{remark}

\begin{remark}
  For $f \colon X \to Y$, the lax monoidal functor
  $f^{*}\colon \Coll_{Y}(\mathcal{V}) \to \Coll_{X}(\mathcal{V})$ is
  given by composition with a morphism of \igpds{}
  $f_{\xFeq} \colon \xFeq_{X} \to \xFeq_{Y}$. Since $\mathcal{V}$ has
  colimits indexed by \igpds{}, this functor has a left adjoint
  $f_{!}$, given by left Kan extension along $f_{\xFeq}$. Moreover,
  since
  $\Alg_{\mathcal{O}}(\Coll_{X}(\mathcal{V})) \to
  \Coll_{X}(\mathcal{V})$ detects limits and sifted colimits for any
  \nsiopd{} $\mathcal{O}$, the functor
  $f^{*} \colon \Alg_{\mathcal{O}}(\Coll_{Y}(\mathcal{V})) \to
  \Alg_{\mathcal{O}}(\Coll_{X}(\mathcal{V}))$ preserves limits and
  sifted colimits, since this is true for
  $f^{*}\colon \Coll_{Y}(\mathcal{V}) \to \Coll_{X}(\mathcal{V})$. If
  $\mathcal{V}$ is presentably symmetric monoidal, then the
  \icat{} $\Alg_{\mathcal{O}}(\Coll_{Y}(\mathcal{V}))$ is presentable,
  since it is equivalent to
  $\Alg_{\mathcal{O}_{\xF,X}}^{\opd}(\mathcal{V})$, which in turn is
  equivalent to the \icat{} of algebras in $\mathcal{V}$ for some
  symmetric \iopd{}. It then follows from the adjoint functor theorem
  that
  $f^{*} \colon \Alg_{\mathcal{O}}(\Coll_{Y}(\mathcal{V})) \to
  \Alg_{\mathcal{O}}(\Coll_{X}(\mathcal{V}))$ has a left adjoint. This implies:
\end{remark}

\begin{cor}\label{opdcocart}
  Let $\mathcal{O}$ be a weakly contractible \nsiopd{} and
  $\mathcal{V}$ a presentably symmetric monoidal \icat{}. Then the
  functor \[\Alg_{\mathcal{O}}(\COLL(\mathcal{V})) \to \mathcal{S}\]
  given by evaluation at $* \in \mathcal{O}_{0}$ is also a cocartesian
  fibration. \qed
\end{cor}
In general the cocartesian morphisms over $f$ are not easily described
in terms of the left Kan extension along the map $f_{\xFeq}\colon
\xFeq_{X}\to \xFeq_{Y}$. However, we can derive a simple description
in the case of monomorphisms of \igpds{}:
\begin{propn}
  Suppose $i \colon X \hookrightarrow Y$ is a monomorphism of \igpds{}.
  Then the left adjoint $i_{!} \colon \Coll_{X}(\mathcal{V}) \to
  \Coll_{Y}(\mathcal{V})$ has a canonical monoidal structure, such
  that composition with $i_{!}$ and $i^{*}$ gives for any \iopd{}
  $\mathcal{O}$ an adjunction
  \[ i_{!} : \Alg_{\mathcal{O}}(\Coll_{X}(\mathcal{V})) \rightleftarrows
    \Alg_{\mathcal{O}}(\Coll_{Y}(\mathcal{V})) : i^{*}. \]
\end{propn}
\begin{proof}
  We will prove this by applying \cite{HA}*{Corollary 7.3.2.12}, which
  requires us to show that for $\Phi, \Psi \in
  \Coll_{X}(\mathcal{V})$, the canonical map
  \[ i_{!}(\Phi \odot_{X} \Psi) \to i_{!}\Phi \odot_{Y} i_{!}\Psi\]
  is an equivalence.

  We first describe $i_{!}\Phi$ more explicitly: For
  $(y_{1},\ldots,y_{n})$ in $Y^{n}_{h \Sigma_{n}}$, we can
  identify the fibre of $X^{n}_{h \Sigma_{n}}$ over this point as
  $X_{y_{1}} \times \cdots \times X_{y_{n}}$ using
  the commutative diagram
  \[
    \begin{tikzcd}
      X_{y_{1}} \times \cdots \times X_{y_{n}} \arrow{r} \arrow{d}
      \drpullback &
      X^{n} \arrow{r} \arrow{d} \drpullback & 
      X^{n}_{h \Sigma_{n}} \arrow{d} \\
      \{(y_{1},\ldots,y_{n}) \}  \arrow{r} & Y^{n} \arrow{r} \arrow{d} \drpullback
      & Y^{n}_{h \Sigma_{n} } \arrow{d} \\
       & * \arrow{r} & B \Sigma_{n},
    \end{tikzcd}
  \]
  where all three squares are cartesian. Hence the fibre of $\xFeq_{X}
  \to \xFeq_{Y}$ at $\tbinom{y_{1},\ldots,y_{n}}{y}$ is equivalent to
  $\prod_{i} X_{y_{i}} \times X_{y}$,
  giving
  \[ i_{!}\Phi \tbinom{y_{1},\ldots,y_{n}}{y} \simeq
    \colim_{(x_{1},\ldots,x_{n},x) \in \prod_{i}X_{y_{i}}
      \times X_{y}} \Phi\tbinom{x_{1},\ldots,x_{n}}{x}.\]
  We can then rewrite the formula for $i_{!}(\Phi \odot_{X}
  \Psi)\tbinom{y_{1},\ldots,y_{n}}{y}$ as
  \begin{multline*}
    i_{!}(\Phi \odot_{X}\Psi)\tbinom{y_{1},\ldots,y_{n}}{y}
     \simeq \colim_{(x_{1},\ldots,x_{n},x) \in \prod_{i}X_{y_{i}}
      \times X_{y}} (\Phi \odot_{X}\Psi)\tbinom{x_{1},\ldots,x_{n}}{x}
    \\
     \simeq \colim_{(x_{1},\ldots,x_{n},x) \in \prod_{i}X_{y_{i}}
      \times X_{y}} \colim_{\mathbf{n} \to \mathbf{m} \to \mathbf{1}}
      \colim_{(x'_{j}) \in X^{m}} \bigotimes_{j} \Phi\tbinom{x_{i} : i
      \in \mathbf{n}_{j}}{x'_{j}} \otimes
      \Psi\tbinom{x'_{1},\ldots,x'_{m}}{x} \\
     \simeq \colim_{\mathbf{n} \to \mathbf{m} \to \mathbf{1}}
      \colim_{(y'_{j}) \in Y^{m}} \colim_{(x_{i}) \in \prod_{i}
      X_{y_{i}}} \colim_{(x'_{j}) \in \prod_{j} X_{y'_{j}}} \colim_{x
      \in X_{y}} \bigotimes_{j} \Phi\tbinom{x_{i} : i
      \in \mathbf{n}_{j}}{x'_{j}} \otimes
      \Psi\tbinom{x'_{1},\ldots,x'_{m}}{x}.
  \end{multline*}
  On the other hand $(i_{!}\Phi \odot_{Y}
  i_{!}\Psi)\tbinom{y_{1},\ldots,y_{n}}{y}$ is equivalent to
  \[ \colim_{\mathbf{n} \to \mathbf{m} \to \mathbf{1}}
      \colim_{(y'_{j}) \in Y^{m}} \colim_{(x_{i}) \in \prod_{i}
      X_{y_{i}}} \colim_{(x'_{j}) \in \prod_{j} X_{y'_{j}}} \colim_{(x''_{j}) \in \prod_{j} X_{y'_{j}}}\colim_{x
      \in X_{y}} \bigotimes_{j} \Phi\tbinom{x_{i} : i
      \in \mathbf{n}_{j}}{x'_{j}} \otimes
    \Psi\tbinom{x''_{1},\ldots,x''_{m}}{x}, \]
  and the canonical map corresponds under these equivalences to the
  map of colimits arising from the diagonal map $\prod_{j} X_{y'_{j}}
  \to \prod_{j} X_{y'_{j}} \times \prod_{j} X_{y'_{j}}$. Since these
  are \igpds{}, this map is cofinal \IFF{} it is an equivalence, which
  holds \IFF{} the spaces $X_{y}$ for $y \in Y$ are either
  contractible or empty, \ie{} \IFF{} $i$ is a monomorphism.
\end{proof}

\begin{cor}\label{cor:i!ff}
  Let $i \colon X \to Y$ be a monomorphism of \igpds{} and
  $\mathcal{O}$ a weakly contractible \nsiopd{}.
  \begin{enumerate}[(i)]
  \item For every $A \in \Alg_{\mathcal{O}}(\Coll_{X}(\mathcal{V}))$,
    the unit morphism $A \to i^{*}i_{!}A$ is an equivalence.
  \item The functor \[i_{!} \colon
    \Alg_{\mathcal{O}}(\Coll_{X}(\mathcal{V})) \to
    \Alg_{\mathcal{O}}(\Coll_{Y}(\mathcal{V}))\] is fully faithful. \qed
  \end{enumerate}
\end{cor}

\begin{remark}
  We will also need a more general version of \cref{sseqthm}, which follows
  by using part (iii) of \cite{symmseq}*{Proposition 3.5.6} instead of
  (vi): If $F \colon \mathcal{V} \to \mathcal{W}$ is a symmetric
  monoidal functor then composition with $F$ induces a morphism of
  \gnsiopds{} $F_{*} \colon \COLL(\mathcal{V}) \to \COLL(\mathcal{W})$, which
  restricts to lax monoidal functors $F_{*}\colon \Coll_{X}(\mathcal{V}) \to
  \Coll_{X}(\mathcal{W})$. These are compatible with the lax monoidal
  functors $f^{*}$ coming from maps of spaces $f \colon X \to Y$: A
  priori the square
  \[
    \begin{tikzcd}
      \Coll_{Y}(\mathcal{V})  \arrow{r}{F_{*}} \arrow{d}{f^{*}} &
      \Coll_{Y}(\mathcal{W}) \arrow{d}{f^{*}} \\
      \Coll_{X}(\mathcal{V})  \arrow{r}{F_{*}}  &
      \Coll_{X}(\mathcal{W})      
    \end{tikzcd}
  \]
  only commutes up to a natural transformation, but this is clearly a
  natural equivalence since both functors are given by composition.
\end{remark}

\section{Algebras for $\infty$-Operads as
  Modules}\label{subsec:algmod}
In this section we define algebras for an enriched \iopd{}
$\mathcal{O}$ as certain right $\mathcal{O}$-modules in
$\COLL(\mathcal{V})$. We first recall the definition of the
non-symmetric \iopd{} for right modules, and prove that this is weakly
contractible, allowing us to apply \cref{wcopdfib}:
\begin{defn}\label{defn:RMopd}
  Let $\mathbf{rm}$ denote the non-symmetric operad for right modules.
  This has two objects, $a$ and $m$, and there is a unique
  multimorphism $(x_{1},\ldots,x_{n}) \to y$ if
  $x_{1} = \cdots = x_{n} = y = a$ ($n = 0$ allowed) or $x_{1}=y=m$
  and $x_{2}=\cdots = x_{n}=a$, and no multimorphisms otherwise. We
  write $\RM \to \Dop$ for the corresponding \nsiopd{}, or in other
  words the category of operators of $\mathbf{rm}$. This has objects
  sequences $(x_{1},\ldots,x_{n})$ with each $x_{i}$ being either $a$
  or $m$, and a morphism
  $(x_{1},\ldots,x_{n}) \to (y_{1},\ldots,y_{m})$ is given by a map
  $\phi \colon [m] \to [n]$ in $\simp$ and multimorphisms
  $(x_{\phi(i-1)+1},\ldots, x_{\phi(i)}) \to y_{i}$ in $\mathbf{rm}$.
\end{defn}

\begin{propn}
  The category $\RM$ is weakly contractible.
\end{propn}
\begin{proof}
  In this proof it is convenient to use the notation
  $(i_{0},\ldots,i_{n})_{\RM}$ for the object of $\RM$ given by the
  sequence $(a,\ldots,a,m,\cdots,m,a,\ldots,a)$ where there are $n$
  copies of $m$ and $i_{t}$ copies of $a$ between the $t$th and
  $(t+1)$th copy of $m$ (and $i_{0}$ before the first and $i_{n}$
  after the last). Define a functor $\mu \colon \Dintop \to \RM$ over $\Dop$ by
  taking $[n]$ to the unique object of the form $[0,\ldots,0]_{\RM} =
  (m,\ldots,m)$ over $[n]$, and determined on morphisms by the
  inert morphisms between these objects. We claim that $\mu$ is
  coinitial, and so in particular a weak homotopy equivalence. To see
  this, it suffices by \cite{HTT}*{Theorem 4.1.3.1} to show that for
  every object $X \in \RM$ the category $(\Dintop)_{/X}$ is weakly
  contractible. But this category has a terminal object: if $X =
  (i_{0},\ldots,i_{n})_{\RM}$ then any morphism $(0,\ldots,0)_{\RM}
  \to X$ factors as an inert morphism followed by the (unique)
  degeneracy $\mu([n]) \to X$. Since $\Dintop$ is weakly contractible
  (for example, because the inclusion $\Dintop \hookrightarrow \Dop$
  is cofinal and $\Dop$ has an initial object), this implies that
  $\RM$ is also weakly contractible.
\end{proof}

\begin{cor}
  The functor \[\Algd^{\opd}_{\RM_{\xF}}(\mathcal{V}) \simeq
    \Alg_{\RM}(\COLL(\mathcal{V})) \to \mathcal{S},\]
  given by evaluation at $() \in \RM_{0}$, is a cartesian
  fibration corresponding to the functor $\mathcal{S} \to \CatI$ that
  takes $X$ to $\Alg_{\RM}(\Coll_{X}(\mathcal{V}))$ and a
  morphism $f \colon X \to Y$ to the functor given by composition with
  the lax monoidal functor $f^{*}\colon \Coll_{Y}(\mathcal{V}) \to
  \Coll_{X}(\mathcal{V})$. \qed
\end{cor}

To define algebras we want to restrict to those modules that are
concentrated in degree $0$, which will be justified by the next proposition.
\begin{defn}
  We say that $\Phi \in \Coll_{X}(\mathcal{V})$ is \emph{concentrated
    in degree $0$} if \[\Phi\binom{x_{1},\ldots,x_{n}}{y} \simeq
  \emptyset\] whenever $n > 0$, where $\emptyset$ denotes the initial
object in $\mathcal{V}$.
\end{defn}

\begin{propn}\label{propn:degzero}
  Let $\mathcal{V}$ be a symmetric monoidal \icat{} compatible with
  colimits indexed by small \igpds{}.
  \begin{enumerate}[(i)]
  \item The functor $Z \colon \Coll_{X}(\mathcal{V}) \to \Fun(X,
  \mathcal{V})$ given by composition with $X \hookrightarrow
  \mathbb{F}^{\simeq}_{X}$ has a fully faithful left adjoint, which identifies
    $\Fun(X, \mathcal{V})$ with the collections
    that are concentrated in degree $0$.
  \item If $M \colon \mathbb{F}^{\simeq}_{X} \to \mathcal{V}$ is
    concentrated in degree $0$, then so is $M \odot_{X} N$ for any $N
    \in \Coll_{X}(\mathcal{V})$.
  \item The composition product induces a right
    $\Coll_{X}(\mathcal{V})$-module structure on the \icat{}
    $\Fun(X, \mathcal{V})$.
  \item For $f \colon X \to Y$, composition with $f$ and the induced
    functor $\xFeq_{X} \to \xFeq_{Y}$ gives a lax $\RM$-monoidal
    functor \[f^{*} \colon (\Fun(Y, \mathcal{V}),
      \Coll_{Y}(\mathcal{V})) \to (\Fun(X, \mathcal{V}),
      \Coll_{X}(\mathcal{V}))\]
  \item Composition with a symmetric monoidal functor $F \colon
    \mathcal{V} \to \mathcal{W}$ gives a lax $\RM$-monoidal functor
    \[F_{*} \colon (\Fun(X, \mathcal{V}), \Coll_{X}(\mathcal{V})) \to
      (\Fun(X, \mathcal{W}), \Coll_{X}(\mathcal{W})).\] If $F$
    preserves colimits indexed by small \igpds{}, then $F_{*}$ is an
    $\RM$-monoidal functor.
  \end{enumerate}
\end{propn}
\begin{proof}
  Part (i) is obvious from the description of $\xF^{\simeq}_{X}$ as
  $\coprod_{n} X^{\times n}_{h\Sigma_{n}} \times X$ and the formula
  for pointwise left Kan extensions, while part (ii) follows
  immediately from the description of composition of horizontal
  morphisms in $\COLL(\mathcal{V})$ in \cref{sseqthm}. Part (iii) then
  holds by combining parts (i) and (ii), and parts (iv) and (v) follow
  by restricting the lax monoidal functors discussed in
  \S\ref{sec:review}.
\end{proof}

\begin{defn}
  Let $\mathcal{O}$ be a $\mathcal{V}$-\iopd{} with space of objects
  $X$, viewed as an associative algebra in
  $\Coll_{X}(\mathcal{V})$. An \emph{$\mathcal{O}$-algebra in
    $\mathcal{V}$} is a
  right $\mathcal{O}$-module in $\Fun(X, \mathcal{V})$. We write
  $\Alg_{\mathcal{O}}(\mathcal{V})$ for the \icat{}
  $\RMod_{\mathcal{O}}(\Fun(X, \mathcal{V}))$ of these right
  modules.
\end{defn}

\begin{remark}\label{rmk:algfun}
  By \cref{propn:degzero}(iv) we see that for $\mathcal{O} \in
  \Opd_{Y}(\mathcal{V})$, composition with $f \colon X \to Y$ gives a
  functor $\Alg_{\mathcal{O}}(\mathcal{V}) \to
  \Alg_{f^{*}\mathcal{O}}(\mathcal{V})$, while composition with a
  symmetric monoidal functor $F \colon \mathcal{V} \to \mathcal{W}$
  gives a functor $\Alg_{\mathcal{O}}(\mathcal{V}) \to
  \Alg_{F_{*}\mathcal{O}}(\mathcal{W})$.
\end{remark}

Since there is always a formula for free modules, with this definition
we immediately get a formula for free algebras over enriched \iopds{}:
\begin{propn}\label{propn:freeopdalg}
  The forgetful functor $U_{\mathcal{O}}\colon \Alg_{\mathcal{O}}(\mathcal{V}) \to \Fun(X,
  \mathcal{V})$ has a left adjoint $F_{\mathcal{O}}$, and the
  endofunctor $U_{\mathcal{O}}F_{\mathcal{O}}$ satisfies
  \[ U_{\mathcal{O}}F_{\mathcal{O}}M(x) \simeq \coprod_{n}
    \colim_{(x_{1},\ldots,x_{n}) \in
      X^{n}_{h \Sigma_{n}}} M(x_{1})\otimes \cdots \otimes
    M(x_{n}) \otimes \mathcal{O}\binom{x_{1},\ldots,x_{n}}{x}.\]
  Moreover, $U_{\mathcal{O}}$ preserves sifted colimits and the
  adjunction is monadic.
\end{propn}
\begin{proof}
  By \cite{HA}*{Corollary 4.2.4.8} the left adjoint $F_{\mathcal{O}}$
  exists, and $U_{\mathcal{O}}F_{\mathcal{O}}(M)$ is given by the
  composition product $M \odot \mathcal{O}$ (with $M$ viewed as a
  symmetric sequence concentrated in degree $0$). Expanding out this
  composition product now gives the formula.
  
  It follows from \cite{HA}*{Proposition 4.2.3.1} that
  $U_{\mathcal{O}}$ detects equivalences and from \cite{HA}*{Corollary
    4.2.3.5} that $\Alg_{\mathcal{O}}(\mathcal{V})$ has sifted
  colimits and $U_{\mathcal{O}}$ preserves these, since the
  composition product preserves sifted colimits in each variable. The
  adjunction is therefore monadic by the monadicity theorem for
  \icats{}, \cite{HA}*{Theorem 4.7.3.5}.
\end{proof}

Applying \cite{enr}*{Proposition A.5.9}, we get:
\begin{cor}
  If $\mathcal{V}$ is a presentably symmetric monoidal \icat{} and
  $\mathcal{O}$ is a $\mathcal{V}$-enriched \iopd{}, then the \icat{}
  $\Alg_{\mathcal{O}}(\mathcal{V})$ is presentable. \qed
\end{cor}

\begin{remark}\label{rmk:ffeseqonalgs}
  Let $F \colon \mathcal{O} \to \mathcal{O}'$ be a morphism of
  $\mathcal{V}$-\iopds{} given on spaces of objects by $f \colon X \to
  Y$, and suppose $f$ is surjective on $\pi_{0}$ and $F$ is fully
  faithful in the sense that all the maps
  \[ \mathcal{O}\tbinom{x_{1},\dots,x_{n}}{y} \to
    \mathcal{O}'\tbinom{f(x_{1}),\dots,f(x_{n})}{f(y)}\] are
  equivalences in $\mathcal{V}$. Then we have a commutative square
  \[
    \begin{tikzcd}
     \Alg_{\mathcal{O}'}(\mathcal{V}) \ar[r, "F^{*}"] \ar[d] & \Alg_{\mathcal{O}}(\mathcal{V})  \ar[d] \\
     \Fun(Y,\mathcal{V}) \ar[r, "f^{*}"'] & \Fun(X,\mathcal{V}),
    \end{tikzcd}
  \]
  where the surjectivity of $f$ implies that the composite functor
  $\Alg_{\mathcal{O}'}(\mathcal{V}) \to \Fun(X,\mathcal{V})$ is a
  monadic right adjoint. Using the formula from
  \cref{propn:freeopdalg} it is easy to see that $F^{*}$ gives an
  equivalence of monads on $\Fun(X,\mathcal{V})$ and so gives an
  equivalence of \icats{} $\Alg_{\mathcal{O}'}(\mathcal{V}) \simeq
  \Alg_{\mathcal{O}}(\mathcal{V})$ by \cite{HA}*{Corollary
    4.7.3.16}. This applies in particular if $\mathcal{O}'$ is the
  completion of $\mathcal{O}$, so by a 2-out-of-3 argument it follows
  that any fully faithful and essentially surjective morphism of
  $\mathcal{V}$-\iopds{} $F \colon \mathcal{O} \to \mathcal{P}$
  induces an equivalence
  \[ \Alg_{\mathcal{O}}(\mathcal{V}) \simeq
    \Alg_{\mathcal{P}}(\mathcal{V})\]
  on \icats{} of algebras in $\mathcal{V}$.
\end{remark}

We end this section by showing that the nullary operations of a
$\mathcal{V}$-\iopd{} $\mathcal{O}$ give a canonical
$\mathcal{O}$-algebra, using the next observation:
\begin{propn}
  $Z \colon \Coll_{X}(\mathcal{V}) \to \Fun(X, \mathcal{V})$ is a
  functor of $\Coll_{X}(\mathcal{V})$-modules.
\end{propn}
\begin{proof}
  By definition of the $\Coll_{X}(\mathcal{V})$-module structure on
  $\Fun(X, \mathcal{V})$, the inclusion
  $\Fun(X, \mathcal{V}) \to \Coll_{X}(\mathcal{V})$ is a functor of
  $\Coll_{X}(\mathcal{V})$-modules. Using \cite{HA}*{Corollary
    7.3.2.7}, this implies that its right adjoint $Z$ is a lax
  $\RM$-monoidal functor. Thus for $M, N \in \Coll_{X}(\mathcal{V})$
  there are natural maps
  \[ Z(M) \odot_{X} N \to Z(M \odot_{X} N);\] by the formula for
  $\odot_{X}$ these maps are equivalences, and so $Z$ is an
  $\RM$-monoidal functor.
\end{proof}

\begin{cor}
  If $\mathcal{O}$ is an associative algebra in
  $\Coll_{X}(\mathcal{V})$ and $M \in \Coll_{X}(\mathcal{V})$ is a right
  $\mathcal{O}$-module, then the restriction
  $Z(M) \in \Fun(Y, \mathcal{V})$ is also a right
  $\mathcal{O}$-module. \qed
\end{cor}

Since an algebra is canonically a right module over itself, this
specializes to:
\begin{cor}\label{nullismod}
  Suppose $\mathcal{O}$ is an algebra in
  $\Coll_{X}(\mathcal{V})$, \ie{} a $\mathcal{V}$-\iopd{} with $X$ as
  space of objects. Then the functor $Z(\mathcal{O}) \colon X \to
  \mathcal{V}$ picking out the nullary operations is canonically a
  right  $\mathcal{O}$-module. \qed
\end{cor}

\section{Comparison with Model Categories of Operad Algebras}\label{subsec:comp}
Let $\mathbf{V}$ be a symmetric monoidal model category (with
cofibrant unit). Then by \cite{HA}*{Proposition 4.1.7.4} the
localization $\mathbf{V}[W^{-1}]$ (with $W$ the class of weak
equivalences) is a symmetric monoidal \icat{}, and the
localization functor $\mathbf{V} \to \mathbf{V}[W^{-1}]$ is symmetric
monoidal when restricted to the cofibrant objects. If $\mathbf{O}$ is
a (levelwise cofibrant) operad in $\mathbf{V}$ then this means we can
also view $\mathbf{O}$ as an operad in $\mathbf{V}[W^{-1}]$. Moreover,
in good cases there is a model structure on the category
$\Alg_{\mathbf{O}}(\mathbf{V})$ of $\mathbf{O}$-algebras in
$\mathbf{V}$. In this section we will give conditions under which the
corresponding \icat{}
$\Alg_{\mathbf{O}}(\mathbf{V})[W_{\mathbf{O}}^{-1}]$ (with
$W_{\mathbf{O}}$ the class of weak equivalences of
$\mathbf{O}$-algebras) is equivalent to the \icat{}
$\Alg_{\mathbf{O}}(\mathbf{V}[W^{-1}])$, defined as in the previous
section. In order to do the comparison in sufficient generality to
cover examples such as symmetric spectra, we do not want to assume
that the unit of the monoidal structure is cofibrant. Instead we
consider model categories with a subcategory of \emph{flat} objects in
the following sense:
\begin{defn}
  Let $\mathbf{V}$ be a symmetric monoidal model category.\footnote{We
    assume that model categories have functorial factorizations.}  A
  \emph{subcategory of flat objects} is a full subcategory
  $\mathbf{V}^{\flat}$ that satisfies the following conditions:
\begin{itemize}
\item $\mathbf{V}^{\flat}$ is a symmetric monoidal subcategory, \ie{}
  the unit is flat and the tensor product of two flat objects is flat,
\item If $X$ is flat and $Y \to Y'$ is a weak equivalence between flat
  objects, then $X \otimes Y \to X \otimes Y'$ is again a weak
  equivalence.
\item All cofibrant objects are flat.
\end{itemize}
\end{defn}
\begin{ex}
  If the unit of $\mathbf{V}$ is cofibrant, then the subcategory
  $\mathbf{V}^{c}$ of cofibrant objects is a subcategory of flat
  objects. 
\end{ex}

\begin{propn}
  Let $\mathbf{V}$ be a symmetric monoidal model category and
  $\mathbf{V}^{\flat}$ a subcategory of flat objects. Then the inclusions
$\mathbf{V}^{c} \hookrightarrow \mathbf{V}^{\flat} \hookrightarrow
\mathbf{V}$ induce equivalences of localizations
\[ \mathbf{V}^{c}[W^{-1}] \isoto \mathbf{V}^{\flat}[W^{-1}] \isoto
  \mathbf{V}[W^{-1}],\] where we denote the collections of weak
equivalences in the subcategories by $W$ in all cases.
\end{propn}
\begin{proof}
  Let $Q \colon \mathbf{V} \to \mathbf{V}$ be a cofibrant
  replacement functor, with a natural weak equivalence $\eta \colon Q
  \to \id$. If $i$ denotes the inclusion $\mathbf{V}^{c}
  \hookrightarrow \mathbf{V}$ then we may view $Q$ as a functor
  $\mathbf{V} \to \mathbf{V}^{c}$ and $\eta$ as a natural
  transformation $i Q \to \id_{\mathbf{V}}$. If $X$ is cofibrant, then
  $\eta_{X} \colon QX \to X$ is a morphism in $\mathbf{V}^{c}$, so we
  may view $\eta i \colon iQi \to i$ as a natural transformation
  $\eta^{c} \colon Qi \to \id_{\mathbf{V}^{c}}$. The functor $Q$
  preserves weak equivalences, and both $\eta$ and $\eta^{c}$ are
  natural weak equivalences. It follows that $Q$ induces a functor
  $\mathbf{V}[W^{-1}] \to \mathbf{V}^{c}[W^{-1}]$ and the
  transformations $\eta$ and $\eta^{c}$ induce transformations that
  exhibit this as an inverse of the functor $\mathbf{V}^{c}[W^{-1}]
  \to \mathbf{V}[W^{-1}]$ induced by $i$. The same argument applies to
  $Q$ restricted to the full subcategory $\mathbf{V}^{\flat}$; the
  functor
  $\mathbf{V}^{\flat}[W^{-1}] \to \mathbf{V}[W^{-1}]$ is therefore an
  equivalence by the 2-of-3 property of equivalences.
\end{proof}

\begin{cor}
  Let $\mathbf{V}$ be a symmetric monoidal model category and
  $\mathbf{V}^{\flat}$ a subcategory of flat objects. Then the \icat{}
  $\mathbf{V}[W^{-1}]$
  inherits a symmetric monoidal structure such that the functor
  $\mathbf{V}^{\flat} \to \mathbf{V}[W^{-1}]$ is symmetric monoidal.
\end{cor}
\begin{proof}
  By assumption, in $\mathbf{V}^{\flat}$ the tensor product is
  compatible with weak equivalences, and so the \icat{}
  $\mathbf{V}[W^{-1}] \simeq \mathbf{V}^{\flat}[W^{-1}]$ inherits a
  symmetric monoidal structure with this property by
  \cite{HA}*{Proposition 4.1.7.4}.
\end{proof}

Using \cref{rmk:algfun}, composition with the
symmetric monoidal functor $\mathbf{V}^{\flat} \to \mathbf{V}[W^{-1}]$
gives a natural functor
\[ \Alg_{\mathbf{O}}(\mathbf{V}^{\flat}) \to
  \Alg_{\mathbf{O}}(\mathbf{V}[W^{-1}]),\]
if $\mathbf{O}$ is a levelwise flat $\mathbf{V}$-operad. Here we can
interpret $\Alg_{\mathbf{O}}(\mathbf{V}^{\flat})$ as the classical
ordinary category of $\mathbf{O}$-algebras in $\mathbf{V}^{\flat}$.

\begin{defn}
  An $S$-coloured operad $\mathbf{O}$ in a symmetric monoidal model category $\mathbf{V}$ is
  called \emph{admissible} if there exists a model structure on
  $\Alg_{\mathbf{O}}(\mathbf{V})$ where a morphism is a weak
  equivalence or a fibration precisely if its underlying morphism in
  $\Fun(S, \mathbf{V})$ is one (\ie{} it is a weak equivalence or
  fibration in $\mathbf{V}$ for each element of $S$).  
\end{defn}

\begin{defn}
  An $S$-coloured $\mathbf{V}$-operad $\mathbf{O}$ is called
  \emph{$\Sigma$-cofibrant} if the unit map $\bbone_{S} \to
  U(\mathbf{O})$ is a cofibration in the projective model structure on
  $\Fun(\mathbb{F}^{\simeq}_{S}, \mathbf{V})$, where $U$ denotes the
  forgetful functor from operads to collections and $\bbone_{S}$ is
  the monoidal unit for the composition product, given by
  \[ \bbone_{S}\tbinom{s_{1},\ldots,s_{n}}{s'} =
    \begin{cases}
      \bbone, & n = 1, s_{1} = s',\\
      \emptyset, & \txt{otherwise},
    \end{cases}
  \]
  where $\bbone$ is the monoidal unit in $\mathbf{V}$.
\end{defn}

\begin{ex}
  A one-coloured $\mathbf{V}$-operad $\mathbf{O}$ is
  $\Sigma$-cofibrant precisely if $\bbone \to \mathbf{O}(1)$ is a
  cofibration, and the object $\mathbf{O}(n)$ is projectively
  cofibrant in $\Fun(B\Sigma_{n}, \mathbf{V})$ for all $n \neq 1$.
\end{ex}

\begin{defn}
  Let $\mathbf{V}$ be a symmetric monoidal model category and
  $\mathbf{V}^{\flat}$ a subcategory of flat objects. We will say that
  a $\mathbf{V}$-operad $\mathbf{O}$ is \emph{flat} if it is enriched
  in the full subcategory $\mathbf{V}^{\flat}$.
\end{defn}

\begin{remark}
  Since cofibrant objects are flat, if $\mathbf{O}$ is
  $\Sigma$-cofibrant then it is flat precisely if in addition the objects of
  (unary) endomorphisms $\mathbf{O}(x,x) \in \mathbf{V}$ are all
  flat.
\end{remark}

By \cite{PavlovScholbachOpd}*{Proposition 6.2}, if $\mathbf{O}$ is an
admissible $\Sigma$-cofibrant $\mathbf{V}$-operad, then cofibrant
$\mathbf{O}$-algebras have cofibrant underlying objects in
$\mathbf{V}$. Since cofibrant objects are in particular flat, if
$\mathbf{O}$ is flat, admissible and $\Sigma$-cofibrant we have a functor
\[ \Alg_{\mathbf{O}}(\mathbf{V})^{c} \to
  \Alg_{\mathbf{O}}(\mathbf{V}^{\flat}) \to
  \Alg_{\mathbf{O}}(\mathbf{V}[W^{-1}]).\]
This takes weak equivalences in $\Alg_{\mathbf{O}}(\mathbf{V})^{c}$ to
equivalences in $\Alg_{\mathbf{O}}(\mathbf{V}[W^{-1}])$, since the
weak equivalences are lifted from the weak equivalences in
$\mathbf{V}$, and so induces a functor of \icats{}
\[ \Alg_{\mathbf{O}}(\mathbf{V})^{c}[W_{\mathbf{O}}^{-1}] \to
  \Alg_{\mathbf{O}}(\mathbf{V}[W^{-1}]),\]
where $W_{\mathbf{O}}$ denotes the collection of weak
equivalences between $\mathbf{O}$-algebras.
\begin{thm}\label{thm:opdalgrect}
  Let $\mathbf{V}$ be a symmetric monoidal model category equipped
  with a subcategory $\mathbf{V}^{\flat}$ of flat objects. If
  $\mathbf{O}$ is a flat admissible
  $\Sigma$-cofibrant $\mathbf{V}$-operad, then the functor
  \[ \Alg_{\mathbf{O}}(\mathbf{V})^{c}[W_{\mathbf{O}}^{-1}] \to
    \Alg_{\mathbf{O}}(\mathbf{V}[W^{-1}])\] is an equivalence of
  \icats{}.
\end{thm}
\begin{proof}
  We follow the proof of \cite{PavlovScholbachOpd}*{Theorem 7.10},
  which in turn is a variant of those of \cite{HA}*{Theorems 4.1.4.4,
    4.5.4.7}.  Let $S$ be the set of objects of $\mathbf{O}$.  The
  right Quillen functor
  $\Alg_{\mathbf{O}}(\mathbf{V}) \to \Fun(S, \mathbf{V})$ induces a
  functor of \icats{}
  $U \colon \Alg_{\mathbf{O}}(\mathbf{V})^{c}[W_{\mathbf{O}}^{-1}] \to
  \Fun(S, \mathbf{V}[W^{-1}])$, which is a right adjoint by
  \cite{MazelGeeQAdj}*{Theorem 2.1}. As $\mathbf{O}$ is
  $\Sigma$-cofibrant, the forgetful functor preserves sifted homotopy
  colimits by \cite{PavlovScholbachOpd}*{Proposition 7.8}. Since it
  also detects weak equivalences, it follows by \cite{HA}*{Theorem
    4.7.3.5} (the monadicity theorem for \icats{}) that $U$ is a
  monadic right adjoint. The same holds for the forgetful functor
  $\Alg_{\mathbf{O}}(\mathbf{V}[W^{-1}]) \to \Fun(S,
  \mathbf{V}[W^{-1}])$ by Proposition~\ref{propn:freeopdalg}, so using
  \cite{HA}*{Corollary 4.7.3.16} we see that to show that the functor
  $\Alg_{\mathbf{O}}(\mathbf{V})^{c}[W_{\mathbf{O}}^{-1}] \to
  \Alg_{\mathbf{O}}(\mathbf{V}[W^{-1}])$ is an equivalence it suffices
  to show that the two associated monads on
  $\Fun(S, \mathbf{V}[W^{-1}])$ have equivalent underlying
  endofunctors. This follows from the formula in
  Proposition~\ref{propn:freeopdalg}, since the $\Sigma_{n}$-orbits
  that appear in the formula for free strict $\mathbf{O}$-algebras are
  homotopy orbits when $\mathbf{O}$ is $\Sigma$-cofibrant.
\end{proof}

The cases to which this result applies are primarily those where
\emph{all} operads are admissible, as more generally we only have
semi-model structure on algebras over $\Sigma$-cofibrant operads.
This includes the following examples, as discussed in 
\cite[\S 7]{PavlovScholbachSymm}:
\begin{enumerate}[(i)]
\item the category $\Set_{\Delta}$ of simplicial sets, equipped with
  the Kan--Quillen model structure,
\item the category $\txt{Top}$ of compactly generated weak Hausdorff
  spaces, equipped with the usual model structure,
\item the category $\txt{Ch}_{k}$ of chain complexes of $k$-vector
  spaces, where $k$ is a field of characteristic $0$ (or more
  generally a ring containing $\mathbb{Q}$), equipped with the
  projective model structure,
\item the category $\txt{Sp}^{\Sigma}$ of symmetric spectra,
  equipped with the positive stable model structure,
\end{enumerate}
In the first three examples the unit is cofibrant, and in the positive
stable model structure on symmetric spectra a suitable subcategory of
flat objects is supplied by the \emph{$S$-cofibrant} objects of
\cite{ShipleyConvenient} (see also \cite{SchwedeSymmSp}*{Chapter 5},
where these are called \emph{flat} objects). 
Note that a $\Sigma$-cofibrant operad in symmetric spectra is
necessarily flat, since the flat objects are the cofibrant objects in
a model structure whose cofibrations include the usual cofibrations.

Specializing to these cases, we have:
\begin{cor}
  \
  \begin{enumerate}[(i)]
  \item Let $\mathbf{O}$ be a $\Sigma$-cofibrant simplicial operad,
    then \[\Alg_{\mathbf{O}}(\Set_{\Delta})[W_{\mathbf{O}}^{-1}] \simeq
    \Alg_{\mathbf{O}}(\mathcal{S}).\]
  \item Let $\mathbf{O}$ be a $\Sigma$-cofibrant topological operad,
    then \[\Alg_{\mathbf{O}}(\txt{Top})[W_{\mathbf{O}}^{-1}] \simeq
    \Alg_{\mathbf{O}}(\mathcal{S}).\]
  \item Let $\mathbf{O}$ be a $\Sigma$-cofibrant dg-operad over a
    field $k$ of characteristic zero, then
    \[\Alg_{\mathbf{O}}(\txt{Ch}_{k})[W_{\mathbf{O}}^{-1}] \simeq
      \Alg_{\mathbf{O}}(\mathcal{D}(k)),\] where $\mathcal{D}(k)$ is
    the derived \icat{} of $k$-modules.
  \item Let $\mathbf{O}$ be a $\Sigma$-cofibrant operad in
    symmetric spectra, then
    \[\Alg_{\mathbf{O}}(\txt{Sp}^{\Sigma})[W_{\mathbf{O}}^{-1}] \simeq
    \Alg_{\mathbf{O}}(\Sp),\] where $\Sp$ is the \icat{} of spectra.
  \end{enumerate}
\end{cor}
\begin{remark}
  The case of simplicial operads was already proved as
  \cite{PavlovScholbachOpd}*{Theorem 7.10}.
\end{remark}

\begin{remark}
  According to Spitzweck's thesis \cite{SpitzweckOpd}*{Theorem 4}, a
  $\mathbf{V}$-operad with a single object that is cofibrant in the
  semi-model structure on one-object operads in $\mathbf{V}$ is
  admissible without further assumptions on $\mathbf{V}$. A version
  for coloured operads does not yet seem to appear in the literature,
  but if this is correct then the comparison of \cref{thm:opdalgrect}
  would apply in general for such cofibrant operads.
\end{remark}

\begin{remark}\label{rmk:semimodel}
  If $\mathbf{O}$ is a $\Sigma$-cofibrant operad in a symmetric
  monoidal model category $\mathbf{V}$, then under much weaker
  assumptions on $\mathbf{V}$ there exists a \emph{semi-model}
  structure on the category $\Alg_{\mathbf{O}}(\mathbf{V})$, by a
  result of Spitzweck~\cite{SpitzweckOpd}*{Theorem 5} in the
  one-object case and White--Yau for coloured operads
  \cite{WhiteYauOpd}*{Theorem 6.3.1}. Using results of
  Cisinski~\cite{CisinskiBook}, White and Yau have recently extended
  the results relating structures in model categories to their
  analogues in \icats{} needed to carry out the proof of
  \cref{thm:opdalgrect} in the setting of semi-model categories, and
  thereby extended the comparison with \iopd{} algebras to the case
  where there is only a semi-model structure on algebras over a
  $\Sigma$-cofibrant operad; see \cite{WhiteYauSmith}*{\S 7.3}.
\end{remark}

\section{Endomorphism $\infty$-Operads}\label{subsec:endo}
The first goal of this subsection is to prove that for any morphism of
\igpds{} $f \colon X \to \mathcal{V}^{\simeq}$ there exists a
corresponding \emph{endomorphism \iopd{}} $\End_{\mathcal{V}}(f)$,
where $\mathcal{V}$ denotes a closed symmetric monoidal \icat{}
compatible with small \igpd{}-indexed colimits. Our strategy for
obtaining these objects is taken from \cite{HinichYoneda}*{\S 6.3} and
uses the construction of endomorphism algebras from \cite{HA}*{\S
  4.7.1}, which we first briefly recall:\footnote{We restate it for
  right instead of left modules.}

Suppose $\mathcal{A}$ is a monoidal \icat{} and $\mathcal{M}$ is
right-tensored over $\mathcal{A}$. An \emph{endomorphism algebra} for
an object $M \in \mathcal{M}$ is an associative algebra $\txt{End}(M)$
in $\mathcal{A}$ and a right $\txt{End}(M)$-module structure on $M$
with the universal property that for any associative algebra $A$ in
$\mathcal{A}$, right $A$-module structures on $M$ are naturally
equivalent to morphisms of associative algebras $A \to \txt{End}(M)$.

By \cite{HA}*{Proposition 4.7.1.30, Theorem 4.7.1.34} there exists a
monoidal \icat{} $\mathcal{A}[M]$ whose objects are pairs $(X \in
\mathcal{A}, M \otimes X \to M \txt{ in } \mathcal{M})$, with the
property that an associative algebra in $\mathcal{A}[M]$ corresponds to
an assocative algebra $A \in \mathcal{A}$ together with a right
$A$-module structure on $M$. An endomorphism algebra for $M$ is thus
precisely a terminal object in $\Alg_{\Dop}(\mathcal{A}[M])$. Since
the terminal object of $\mathcal{A}[M]$ has a unique algebra structure
if it exists, we have:
\begin{propn}[\cite{HA}*{Corollary 4.7.1.40}]\label{propn:termendalg}
  If $\mathcal{A}[M]$ has a terminal object $(A, M \otimes A \to M)$
  then $A$ is the underlying object of an endomorphism algebra for
  $M$.\qed
\end{propn}
We also note that by construction the forgetful functor
$\mathcal{A}[M] \to \mathcal{A}$ is a right fibration, corresponding
to the functor
\[ A \mapsto \Map_{\mathcal{M}}(M \otimes A, M).\]

In the case of $\Coll_{X}(\mathcal{V})$ and its right module
$\Fun(X, \mathcal{V})$ we can explitly identify this functor:
\begin{propn}\label{propn:endrepr}
  For $M \in \Fun(X, \mathcal{V})$ and $S \in \Coll_{X}(\mathcal{V})$ there is a natural
  equivalence
  \[ \Map_{\Fun(X, \mathcal{V})}(M \odot S, M) \simeq
    \Map_{\Coll_{X}(\mathcal{V})}(S,
    \End_{\mathcal{V}}(M)),\]
  where $\End_{\mathcal{V}}(M) \colon \xFeq_{X} \to \mathcal{V}$ is the functor given by
  \[ \End_{\mathcal{V}}(M)\tbinom{x_{1},\ldots,x_{n}}{x} \simeq \MAP_{\mathcal{V}}(M(x_{1}) \otimes
    \cdots \otimes M(x_{n}), M(x)),\]
  with $\MAP_{\mathcal{V}}$ denoting the internal Hom in $\mathcal{V}$.
\end{propn}
\begin{proof}
  Since $X$ is an \igpd{}, the twisted arrow \icat{} $\Tw(X)$ is
  equivalent to $X$, and so \cite{freepres}*{Proposition 5.1} yields
  a natural equivalence
  \[ \Map_{\Fun(X, \mathcal{V})}(M \odot S, M) \simeq \lim_{x \in X}
    \Map_{\mathcal{V}}((M \odot S)(x), M(x)).\]
  Now the description of $M \odot S$ from
  Proposition~\ref{propn:freeopdalg} shows that this is naturally
  equivalent to 
\[    \lim_{x \in X} \Map_{\mathcal{V}}\left(\coprod_{n}
    \colim_{(x_{1},\ldots,x_{n}) \in X^{n}_{h
      \Sigma_{n}}}  M(x_{1}) \otimes
  \cdots \otimes M(x_{n}) \otimes S\tbinom{x_{1},\ldots,x_{n}}{x}, M(x)\right).\]
Taking the limit out and applying the universal property of
$\txt{MAP}$, this becomes
\[ \lim_{x \in X} \left( \prod_{n} \lim_{(x_{1},\ldots,x_{n}) \in
      X^{n}_{h \Sigma_{n}}}
    \Map_{\mathcal{V}}(S\tbinom{x_{1},\ldots,x_{n}}{x}, \txt{End}(M)\tbinom{x_{1},
    \ldots, x_{n}}{x})\right).\]
We can now combine the limits to get a limit over $\coprod_{n} X \times
X^{n}_{h\Sigma_{n}} \simeq \mathbb{F}^{\simeq}_{X}$, \ie{}
\[ \lim_{\xi \in \mathbb{F}^{\simeq}_{X}} \Map_{\mathcal{V}}(S(\xi),
  \End_{\mathcal{V}}(M)(\xi)).\]
Applying \cite{freepres}*{Proposition 5.1} once more now identifies this
limit (since
$\mathbb{F}^{\simeq}_{X}$ is again an \igpd{}) with
$\Map_{\Fun(\mathbb{F}^{\simeq}_{X}, \mathcal{V})}(S,
\End_{\mathcal{V}}(M))$, as required.
\end{proof}

\begin{cor}
  For any $M \colon X \to \mathcal{V}$, the \icat{}
  $\Coll_{X}(\mathcal{V})[M]$ has a terminal object.
\end{cor}
\begin{proof}
  By Proposition~\ref{propn:endrepr}, the functor
  $\Coll_{X}(\mathcal{V})^{\op} \to \mathcal{S}$ corresponding to
  the right fibration \[\Coll_{X}(\mathcal{V})[M] \to \Coll_{X}(\mathcal{V})\] is represented by the
  object $\End_{\mathcal{V}}(M)$. This implies that we have an equivalence
  \[ \Coll_{X}(\mathcal{V})[M] \simeq \Coll_{X}(\mathcal{V})_{/\End_{\mathcal{V}}(M)}.\]
  Since the right-hand side clearly has a terminal object, this
  completes the proof.
\end{proof}

Applying Proposition~\ref{propn:termendalg}, we get:
\begin{cor}\label{cor:endopd}
  For any $M \in \Fun(X, \mathcal{V})$ there exists an endomorphism
  \iopd{} $\End_{\mathcal{V}}(M)$ in $\Opd_{X}(\mathcal{V}) \simeq
  \Alg_{\Dop}(\Coll_{X}(\mathcal{V}))$ whose underlying $X$-collection is
  \[ \End_{\mathcal{V}}(M)\tbinom{x_{1},\ldots,x_{n}}{y} \simeq
    \txt{MAP}(M(x_{1})\otimes \cdots \otimes M(x_{n}), M(y)).\]
  This has the universal property that, for any $\mathcal{O} \in
  \Opd_{X}(\mathcal{V})$, morphisms $\mathcal{O} \to
  \End_{\mathcal{V}}(M)$ in $\Opd_{X}(\mathcal{V})$
  correspond to $\mathcal{O}$-algebra structures on $M$, \ie{} there
  is natural equivalence
  \[ \Map_{\Opd_{X}(\mathcal{V})}(\mathcal{O}, \End_{\mathcal{V}}(M))
    \simeq \Alg_{\mathcal{O}}(\mathcal{V})_{M}^{\simeq},\]
  where the right-hand side denotes the underlying \igpd{} of the
  fibre of $\Alg_{\mathcal{O}}(\mathcal{V}) \to \Fun(X, \mathcal{V})$
  at $M$.
\end{cor}

\begin{remark}
  For $X \simeq *$, so that the functor $* \to \mathcal{V}$ picks out
  an object $v$ of $\mathcal{V}$, we get an \icatl{} analogue of the classical
  endomorphism operad: $\End_{\mathcal{V}}(v)$
  is a one-object $\mathcal{V}$-\iopd{} with underlying symmetric
  sequence
 \[\End_{\mathcal{V}}(v)(n) \simeq
   \MAP_{\mathcal{V}}(v^{\otimes n}, v).\]
 If $\mathcal{O}$ is a one-object $\mathcal{V}$-\iopd{}, the universal
 property says that an
 $\mathcal{O}$-algebra structure on $v$ is equivalent to a morphism of
 one-object \iopds{} $\mathcal{O} \to \End_{\mathcal{V}}(v)$.
\end{remark}

\begin{ex}\label{nullend}
  By \cref{nullismod}, if $\mathcal{O}$ is any $\mathcal{V}$-\iopd{}
  with space of objects $X$, then the functor $Z(\mathcal{O}) \colon X
  \to \mathcal{V}$ picking out the nullary operations is canonically a
  right $\mathcal{O}$-module. This corresponds to a canonical morphism
  of $\mathcal{V}$-\iopds{} $\mathcal{O} \to \End(Z(\mathcal{O}))$,
  given by maps
  \[ \mathcal{O}\tbinom{x_{1},\ldots,x_{n}}{y} \to
    \MAP_{\mathcal{V}}(Z(\mathcal{O})(x_{1}) \otimes \cdots \otimes
    Z(\mathcal{O})(x_{n}), Z(\mathcal{O})(y)),\]
  adjoint to the composition maps
  \[ \mathcal{O}\tbinom{}{x_{1}} \otimes \cdots \otimes
    \mathcal{O}\tbinom{}{x_{n}} \otimes
    \mathcal{O}\tbinom{x_{1},\ldots,x_{n}}{y} \to
    \mathcal{O}\tbinom{}{y} \]
  for $\mathcal{O}$.
\end{ex}

We now observe that the endomorphism algebras are compatible with the
lax monoidal functors $f^{*} \colon \Coll_{Y}(\mathcal{V}) \to
\Coll_{X}(\mathcal{V})$ induced by morphisms of \igpds{}
$f \colon X \to Y$:
\begin{propn}
  For $f \colon X \to Y$ a morphism in $\mathcal{S}$ and $M \colon Y
  \to \mathcal{V}$, there is a natural equivalence of $\RM$-algebras
  \[f^{*}(M, \End_{\mathcal{V}}(M))
  \isoto (f^{*}M, \End_{\mathcal{V}}(f^{*}M)).\]
\end{propn}
\begin{proof}
  If $\mathcal{O}$ is a $\mathcal{V}$-\iopd{} with $Y$ as space of
  objects, the lax monoidal functor $f^{*}$ induces a a natural
  functor
  \[ \Alg_{\mathcal{O}}(\mathcal{V}) \to
    \Alg_{f^{*}\mathcal{O}}(\mathcal{V}),\] given on the underlying
  functors to $\mathcal{V}$ by composition with $f$. Applying this to
  the $\mathcal{V}$-\iopd{} $\End_{\mathcal{V}}(M)$ and the canonical
  $\End_{\mathcal{V}}(M)$-algebra structure on $M$, we obtain an
  $f^{*}\End_{\mathcal{V}}(M)$-algebra structure on
  $f^{*}M = M \circ f$. By the universal property of endomorphism
  \iopds{} this corresponds to a morphism of \iopds{}
  $f^{*}\End_{\mathcal{V}}(M) \to \End_{\mathcal{V}}(f^{*}M)$ . Using
  the explicit description of the underlying collection of
  $\End_{\mathcal{V}}(M)$ in terms of internal Homs we see that this
  is an equivalence.
\end{proof}

There exists a universal functor from an \igpd{} to $\mathcal{V}$,
namely the inclusion $\mathcal{V}^{\simeq} \to \mathcal{V}$ of the
underlying \igpd{} of $\mathcal{V}$. Our construction does not apply
directly to this, since the \igpd{} $\mathcal{V}^{\simeq}$ is not
small. However, by passing to a larger universe we can define a
universal endomorphism \iopd{} for $\mathcal{V}$:
\begin{defn}
  Let $\mathcal{V}$ be a large closed symmetric monoidal \icat{}
  compatible with colimits indexed by small \igpds{}. By
  \cite{HA}*{Proposition 4.8.1.10} there is a very large presentable
  \icat{} $\widehat{\mathcal{V}}$ compatible with large colimits, with
  a fully faithful symmetric monoidal functor
  $\mathcal{V} \hookrightarrow \widehat{\mathcal{V}}$ that preserves
  colimits over small \igpds{}. Let $\Opd(\widehat{\mathcal{V}})$ be
  the \icat{} of $\widehat{\mathcal{V}}$-enriched \iopds{} with
  potentially large spaces of objects.
\end{defn}

\begin{remark}
  Since $\mathcal{V}$ is a symmetric monoidal full subcategory of
  $\widehat{\mathcal{V}}$, we can regard $\Opd(\mathcal{V})$ as a full
  subcategory of $\Opd(\widehat{\mathcal{V}})$, containing precisely
  those $\widehat{\mathcal{V}}$-\iopds{} whose spaces of objects are
  small and whose objects of
  multimorphisms all lie in the full subcategory
  $\mathcal{V}$. Similarly, for any $\mathcal{O} \in
  \Opd_{X}(\mathcal{V})$  we have a pullback square
  \[
    \begin{tikzcd}
     \Alg_{\mathcal{O}}(\mathcal{V}) \ar[r,hookrightarrow] \ar[d] & \Alg_{\mathcal{O}}(\widehat{\mathcal{V}})  \ar[d] \\
     \Fun(X,\mathcal{V}) \ar[r,hookrightarrow] & \Fun(X,\widehat{\mathcal{V}})
    \end{tikzcd}
  \]
  where the horizontal maps are fully faithful.
\end{remark}

\begin{defn}
  Let $i \colon \mathcal{V}^{\simeq} \to \widehat{\mathcal{V}}$ denote
  the inclusion of the space of objects in the full subcategory
  $\mathcal{V}$. Applying \cref{cor:endopd} in the enlarged universe
  to $i$, we get an endomorphism $\widehat{\mathcal{V}}$-\iopd{}
  \[ \overline{\mathcal{V}} := \End_{\widehat{\mathcal{V}}}(i).\]
  The formula for its multimorphism objects implies that we can regard
  $\overline{\mathcal{V}}$ as a (large) $\mathcal{V}$-\iopd{} (\ie{} an object
  in the \icat{} $\Opd_{\mathcal{V}^{\simeq}}(\mathcal{V}) \subseteq
  \Opd_{\mathcal{V}^{\simeq}}(\widehat{\mathcal{V}})$, which makes
  sense also when the \igpd{} of objects is large). Moreover, for any
  map of \igpds{} $M \colon X \to \mathcal{V}^{\simeq}$ where $X$ is
  small, we can regard
  \[ \End_{\widehat{\mathcal{V}}}(i\circ M) \simeq M^{*}\overline{\mathcal{V}}\]
  as an object of $\Opd_{X}(\mathcal{V})$.
\end{defn}

\begin{lemma}
  For any map $M \colon X \to \mathcal{V}^{\simeq}$ where $X$ is a
  small \igpd{}, we have a canonical equivalence
  \[ M^{*}\overline{\mathcal{V}} \simeq \End_{\mathcal{V}}(M).\]
\end{lemma}
\begin{proof}
  For $\mathcal{O} \in \Opd_{X}(\mathcal{V})$ we have a natural
  equivalence
  \[ \Alg_{\mathcal{O}}(\mathcal{V})_{M} \isoto \Alg_{\mathcal{O}}(\widehat{\mathcal{V}})_{iM}.\]
  Using the universal property of the endomorphism objects this
  corresponds to a natural equivalence
  \[ \Map(\mathcal{O}, M^{*}\overline{\mathcal{V}}) \isoto
    \Map(\mathcal{O}, \End_{\mathcal{V}}(M)),\] and so an equivalence
  $M^{*}\overline{\mathcal{V}} \isoto \End_{\mathcal{V}}(M)$, as
  required.
\end{proof}

Let $U$ denote the canonical $\overline{\mathcal{V}}$-algebra
structure on $i$, which we can regard as an object of
$\Alg_{\overline{\mathcal{V}}}(\mathcal{V})$. For every map
$M \colon X \to \mathcal{V}^{\simeq}$  with $X$ small, the pullback
$M^{*}U$ is then
the canonical $\End_{\mathcal{V}}(M)$-algebra structure on $M$, which
leads to the following:
\begin{thm}\label{thm:oVmapeq}
  For any small $\mathcal{V}$-\iopd{}
  $\mathcal{O}$, the morphism of \igpds{}
  \[\Map_{\Opd(\widehat{\mathcal{V}})}(\mathcal{O}, \overline{\mathcal{V}}) \to
    \Alg_{\mathcal{O}}(\mathcal{V})^{\simeq},\]
  which takes $\phi \colon \mathcal{O} \to \overline{\mathcal{V}}$ to
  $\phi^{*}U \in \Alg_{\mathcal{O}}(\mathcal{V})^{\simeq}$, is an
  equivalence.
\end{thm}
\begin{proof}
  Let $X$ be the space of objects of $\mathcal{O}$. Then we have a
  commutative triangle of \igpds{}
  \[
    \begin{tikzcd}
    \Map_{\Opd(\widehat{\mathcal{V}})}(\mathcal{O}, \overline{\mathcal{V}})
    \arrow{rr} \arrow{dr} & & \Alg_{\mathcal{O}}(\mathcal{V})^{\simeq}
    \arrow{dl} 
    \\
    & \Map(X, \mathcal{V}).
  \end{tikzcd}
\]
  It suffices to show that we have an equivalence on the fibres over
  each map $M \colon X \to \mathcal{V}$. But we have an equivalence
  between $\Map_{\Opd(\widehat{\mathcal{V}})}(\mathcal{O},
  \overline{\mathcal{V}})_{M}$ and
  \[\Map_{\Opd_{X}(\mathcal{V})}(\mathcal{O},
    M^{*}\overline{\mathcal{V}}) \simeq
    \Map_{\Opd_{X}(\mathcal{V})}(\mathcal{O}, \End_{\mathcal{V}}(M)),\]
  under which the map to
  $\Alg_{\mathcal{O}}(\mathcal{V})^{\simeq}_{M}$ is equivalent to that
  taking $\phi \colon \mathcal{O} \to \End_{\mathcal{V}}(M)$ to
  $\phi^{*}$ applied to the canonical $\End_{\mathcal{V}}(M)$-algebra
  structure on $M$. This is an equivalence by the universal property
  of the endomorphism algebra.
\end{proof}

\begin{remark}
  In \cite{enropd} we constructed a natural tensoring of
  $\mathcal{V}$-\iopds{} over \icats{}. This induces an enrichment in
  \icats{}, given by
  \[ \Map_{\CatI}(\mathcal{C}, \Alg_{\mathcal{O}}(\mathcal{P})) \simeq
    \Map_{\OpdI^{\mathcal{V}}}(\mathcal{C} \otimes \mathcal{O},
    \mathcal{P}).\] For $\mathcal{P} = \overline{\mathcal{V}}$, we can
  use \cref{thm:oVmapeq} to identify the \icat{}
  $\Alg_{\mathcal{O}}(\overline{\mathcal{V}})$ with the Segal space
  $\Alg_{\Delta^{\bullet} \otimes
    \mathcal{O}}(\mathcal{V})^{\simeq}$. We expect that this should in
  fact be equivalent to the \icat{} $\Alg_{\mathcal{O}}(\mathcal{V})$,
  but to prove we need a better understanding of the tensoring of
  $\mathcal{V}$-\iopds{} and \icats{}. As this was defined rather
  inexplicitly in \cite{enropd}, we suspect that this requires setting
  up a new definition of enriched \iopds{} where the tensoring can be
  described more concretely.
\end{remark}

In the case where $\mathcal{V}$ is the \icat{} $\mathcal{S}$ of
spaces, we can identify $\overline{\mathcal{S}}$ explicitly:
\begin{propn}
  Let $\mathcal{S}^{\times}$ denote the symmetric monoidal \icat{}
  given by the cartesian product in
  $\mathcal{S}$, viewed as an $\mathcal{S}$-enriched \iopd{}.
  There is an equivalence $\mathcal{S}^{\times} \isoto
  \overline{\mathcal{S}}$.
\end{propn}
\begin{proof}
  For $X \in \mathcal{S}$, we have
  $Z(\mathcal{S}^{\times})(X) \simeq \Map_{\mathcal{S}}(*, X) \simeq
  X$, and the functor
  $Z(\mathcal{S}^{\times}) \colon \mathcal{S}^{\simeq} \to
  \mathcal{S}$ is the inclusion of the underlying \igpd{}.
  Hence, thinking of $\overline{\mathcal{S}}$ as an endomorphism
  object for \iopds{} enriched in large spaces, by
  \cref{nullend} there is a
  canonical morphism
  $\mathcal{S}^{\times} \to \overline{\mathcal{S}}$. This
  is given by equivalences
  \[ \mathcal{S}^{\times}\tbinom{X_{1},\ldots,X_{n}}{Y} \isoto
    \Map_{\mathcal{S}}(X_{1}\times \cdots \times X_{n}, Y),\]
  and so it is an equivalence of $\mathcal{S}$-\iopds{}.
\end{proof}
\begin{remark}
  It follows that for $\mathcal{O}$ an $\mathcal{S}$-\iopd{}, the
  \igpd{} $\Alg_{\mathcal{O}}(\mathcal{S})^{\simeq}$ in our sense is
  equivalent to $\Map_{\Opd(\widehat{\mathcal{S}})}(\mathcal{O},
  \mathcal{S}^{\times})$. This is the underlying \igpd{} of the
  \icat{} of $\mathcal{O}$-algebras in $\mathcal{S}$ defined in
  \cite{HA}, so for $\mathcal{S}$-enriched \iopds{} our notion of
  $\mathcal{O}$-algebras agrees with that of \cite{HA}, at least on
  the level of \igpds{}.
\end{remark}

\end{document}